\documentclass[a4paper]{article}

\usepackage[margin=1in]{geometry}

\usepackage[T1]{fontenc}
\newcommand{\changefont}[3]{
\fontfamily{#1} \fontseries{#2} \fontshape{#3} \selectfont}

\changefont{ptm}{m}{n}

\usepackage{setspace} 

\usepackage{graphicx}

\usepackage{amsfonts,amsmath}
\usepackage{amssymb}
\usepackage{graphics}
\usepackage{mathrsfs}
 
\usepackage{color}

\newtheorem{theorem}{Theorem}[section]

\newtheorem{lemma}{Lemma}[section]

\newtheorem{definition}{Definition}[section]

\long\def\symbolfootnote[#1]#2{\begingroup%
\def\thefootnote{\fnsymbol{footnote}}\footnote[#1]{#2}\endgroup} 

\begin{document}

\begin{center}
\Large \textbf{Anti-periodic Solutions of Quasilinear Impulsive Systems with Piecewise Constant Argument of Generalized Type}
\end{center}

\begin{center}
\normalsize \textbf{Fatma Tokmak Fen$^{a,}\symbolfootnote[1]{Corresponding Author  E-mail: fatmatokmak@gazi.edu.tr}$}, \, \textbf{Mehmet Onur Fen$^{b}$}, \, \textbf{Eman Bak{\i}r Ahmad Ajm$^{c}$} \\
\vspace{0.2cm}
\textit{\textbf{\footnotesize$^a$Department of Mathematics, Gazi University, 06560 Ankara, Turkey}} \\
\vspace{0.0cm}
\textit{\textbf{\footnotesize$^b$Department of Mathematics, TED University, 06420 Ankara, Turkey}} \\
\vspace{0.0cm}
\textit{\textbf{\footnotesize$^c$Graduate School of Natural and Applied Sciences, Gazi University, 06500 Ankara, Turkey}} \\
\vspace{0.0cm}
\end{center}

\vspace{0.3cm}

\begin{center}
\textbf{Abstract}
\end{center}

\noindent\ignorespaces
This study is devoted to anti-periodic solutions of quasilinear impulsive systems with piecewise constant argument of generalized type. We rigorously prove the existence and uniqueness of an anti-periodic solution making use of the Banach fixed point theorem. The presence of a term with piecewise constant argument is the main novelty. Appropriate examples which support the theoretical findings are provided.

\vspace{0.2cm}
 
\noindent\ignorespaces \textbf{Keywords:} impulsive systems, piecewise constant argument of generalized type, anti-periodic solution

\vspace{0.2cm}

\noindent\ignorespaces \textbf{2020 Mathematics Subject Classification:} 34K45, 34K60

\vspace{0.6cm}

\section{Introduction} \label{secintro}
Studies on differential equations with piecewise constant argument of generalized type was initiated by Akhmet \cite{Akhmet07,Akhmet08}. Such arguments generalize the ones based on the greatest integer function  utilized in the papers \cite{Cooke84,Wiener88}. Impulsive differential equations, on the other hand, are useful to model processes in which sudden changes occur \cite{Samoilenko95,Akhmet10}. The reader is referred to the book \cite{Stamova16a} for applications of systems with impulses. In the present study, we focus on impulsive systems with piecewise constant argument of generalized type and demonstrate that bounded as well as anti-periodic solutions take place in the dynamics under sufficient conditions.

Suppose that $\theta=\left\{\theta_j\right\}_{j\in\mathbb Z}$ is a strictly increasing sequence of real numbers such that $\left|\theta_j\right| \to \infty$ as $\left|j\right|\to\infty$. We will denote by $\mathcal{PC}(\theta)$ the set of all functions $\varphi:\mathbb R \to \mathbb R^n$ such that for each $j\in\mathbb Z$,  the right and left limits of $\varphi$ both exist at $\theta_j$, $\varphi$ is continuous on the interval $(\theta_{j},\theta_{j+1})$, and it is left continuous at $\theta_j$. The main purpose of the present study is the investigation of anti-periodic solutions of impulsive systems with piecewise constant argument of generalized type in the form
\begin{eqnarray} \label{mainsyst}
\begin{aligned}
& z'(t) = A z(t) + f (z(t), z(\gamma(t)))+g(t), \ t\neq \theta_j,    \\
& \Delta z \big{|}_{t=\theta_j}= Bz(\theta_j) + h(z(\theta_j)),  
\end{aligned} 
\end{eqnarray}
where $t\in\mathbb R$, $j\in\mathbb Z$, $A\in\mathbb R^{n \times n}$ and $B\in\mathbb R^{n \times n}$ are constant commutative matrices, the functions $f:\mathbb R^n \times \mathbb R^n \to \mathbb R^n$ and $h:\mathbb R^n \to \mathbb R^n$ are continuous in all their arguments, the function $g:\mathbb R \to \mathbb R^n$ belongs to $\mathcal{PC}(\theta)$, $\Delta z \big{|}_{t=\theta_j}=z(\theta_j+)-z(\theta_j)$, and $z(\theta_j+)=\lim\limits_{t \to \theta_j^+} z(t)$. The piecewise constant argument $\gamma(t)$ is defined via the equation $\gamma(t) = \zeta_j$ for $\theta_j \leq t < \theta_{j+1}$, $j\in\mathbb Z$, in which $\left\{\zeta_j\right\}_{j\in\mathbb Z}$ is a sequence of real numbers satisfying $\theta_j\leq \zeta_j\leq \theta_{j+1}$ for each $j\in\mathbb Z$. We additionally suppose that there exist a natural number $p$ and a real number $\omega>0$ such that $\theta_{j+p}=\theta_j+\omega$ and $\zeta_{j+p}=\zeta_j+\omega$ for every $j\in\mathbb Z$.

The existence and uniqueness of an anti-periodic solution of impulsive retarded functional differential equation was discussed in paper \cite{Afonso17}. The jump equation of the model handled in \cite{Afonso17} does not contain a nonlinear term. However, our model contains a term with piecewise constant argument of generalized type and there is a nonlinear term in the jump equation. On the other hand, the existence of anti-periodic solutions in a class of first order impulsive functional differential equations of the form
\begin{equation} \label{liusystem} 
\begin{aligned}
  & x'(t)+a(t)x(t)=f(t,x(t),x(\alpha_1(t)), \ldots, x(\alpha_n(t))), \ t\neq \theta_j  \\
  & \Delta x |_{t=\theta_j} = J_j(x(\theta_j))
\end{aligned}
\end{equation}
was discussed in the studies \cite {Liu12,Liu13} utilizing the inequalities $x(x+J_j(x)) \geq 0$, $J_j(x)(2x+J_j(x)) \leq 0$, $x J_j(x) \geq 0$. Our results are different compared to \cite{Liu12} and \cite{Liu13} because of the following reasons.
\begin{itemize}
	\item [\textbf{i.}] The model (\ref{liusystem}), which was taken into account in the studies \cite {Liu12,Liu13}, is one dimensional. However, there is no restriction in the dimension of system (\ref{mainsyst}). 
	\item [\textbf{ii.}] The functions $\alpha_k$, $k=1,2,\ldots,n$, in  (\ref{liusystem}) are assumed to be of class $C^1$ on $\mathbb R$ in \cite {Liu12,Liu13}. But this is not true for the piecewise constant argument $\gamma(t)$ in (\ref{mainsyst}) owing to possible discontinuities at $\theta_j$, $j\in\mathbb Z$.
	\item [\textbf{iii.}] The Banach fixed point theorem is utilized and the uniqueness of the anti-periodic solution is  obtained in the present study. In the studies \cite{Liu12,Liu13}, Schaefer fixed point theorem is the main tool and uniqueness is not discussed.
\item [\textbf{iv.}] To obtain the existence and uniqueness of an	anti-periodic solution we do not use inequalities such as $x(x+J_j(x)) \geq 0$, $J_j(x)(2x+J_j(x)) \leq 0$, and $x J_j(x) \geq 0$. Example $1$ provided in Section \ref{secexamples} reveals this distinctiveness.
	\end{itemize}
	
	The rest of the paper is organized as follows. In Section, \ref{secprelim} we propose some sufficient conditions for the existence and uniqueness of bounded and anti-periodic solutions. Moreover, some properties of the matriciant of linear impulsive system associated with (\ref{mainsyst}) are discussed. Section \ref{secbddsol}, on the other hand, is devoted to bounded solutions. We provide our main result on the existence and uniqueness of anti-periodic solutions in Section \ref{secantiper}. Section \ref{secexamples} is concerned with appropriate examples supporting the novel theoretical results. Finally, some concluding remarks are given in Section \ref{secconcremark}.

\section{Preliminaries} \label{secprelim}

A function $\phi \in \mathcal{PC}(\theta)$ is said to be a solution of system (\ref{mainsyst}) if:
\begin{itemize}
	\item [\textbf{i.}] The derivative $\phi'(t)$ exists at each $t\in\mathbb R$ with the possible exception of the points $\theta_j$, $j\in\mathbb Z$, where the one-sided derivatives exist;
	\item [\textbf{ii.}] $\phi(t)$ satisfies the differential equation in (\ref{mainsyst}) on each interval $(\theta_j, \theta_{j+1})$, $j\in\mathbb Z$, and this equation holds for the right derivative of $\phi(t)$ at the points $\theta_j$, $j\in\mathbb Z$; 
	\item [\textbf{iii.}] $\phi(t)$ satisfies the jump equation in (\ref{mainsyst}) for each $j\in\mathbb Z$.
\end{itemize}

Throughout the paper we make use of the Euclidean norm for vectors and the spectral norm for square matrices. The following assumptions on system (\ref{mainsyst}) are required.
\begin{itemize}
\item[\textbf{(A1)}]  $\det(I+B)\neq 0$, where $I$ is the $n\times n$ identity matrix.
\item[\textbf{(A2)}] The real parts of all eigenvalues of the matrix $A+\displaystyle \frac{p}{\omega}\, \textrm{Log}(I+B)$ are negative.
\item[\textbf{(A3)}] There exist positive numbers $M_f$, $M_g$, and $M_h$ such that  
 $\sup\limits_{x\in\mathbb R^n, \, y\in\mathbb R^n} \left\|f(x,y)\right\|\leq M_f$, 
  $\sup\limits_{t\in\mathbb R} \left\|g(t)\right\|\leq M_g$,
  and
  $\sup\limits_{x\in\mathbb R^n} \left\|h(x)\right\|\leq M_h$.
\item[\textbf{(A4)}] There exist positive numbers $L_1$, $L_2$, and $L_h$ such that
\subitem \textbf{i.} $\left\|f(x_1,y)-f(x_2,y)\right\| \leq L_1 \left\|x_1-x_2\right\| $ for every $x_1,x_2,y\in\mathbb R^n$,
\subitem \textbf{ii.} $\left\|f(x,y_1)-f(x,y_2)\right\| \leq L_2 \left\|y_1-y_2\right\|$ for every $x,y_1,y_2\in\mathbb R^n$,
\subitem \textbf{iii.} $\left\|h(x_1) - h(x_2)\right\| \leq L_h \left\|x_1-x_2\right\| $ for every $x_1,x_2 \in \mathbb R^n$.
\item[\textbf{(A5)}] $K \displaystyle\left(  \frac{L_1 + L_2}{\alpha}  +\frac{L_h p }{1-e^{-\alpha \omega}} \right)<1$.
\end{itemize}

In the sequel, the number of the terms of the sequence $\{\theta_j\}_{j\in\mathbb Z}$ belonging to an interval $J$ of the real axis will be denoted by $i(J)$. It can be shown that the inequality
\begin{eqnarray*} \label{countingineq}
-p<i([a,b)) - \displaystyle \frac{p}{\omega} (b-a) \leq p  
\end{eqnarray*}
is fulfilled for any real numbers $a$ and $b$ with $a<b$. Moreover,
\begin{equation} \label{propertyii}
i([a+\omega,b+\omega)) = i([a,b)).
\end{equation}

Let us denote by $U(t,\tau)$ the matriciant of the linear homogeneous impulsive system
\begin{eqnarray} \label{linearsystem}
&& z'(t) = A z(t), \ t\neq \theta_j, \nonumber \\
&& \Delta z \big{|}_{t=\theta_j}= Bz(\theta_j),
\end{eqnarray}
in which the matrices $A$, $B$ and the impulse moments $\theta_j$, $j\in\mathbb Z$, are the same with the ones used in (\ref{mainsyst}). Under the assumption $(A1)$ we have 
\begin{eqnarray*}
U(t,\tau)=e^{A(t-\tau)} (I+B)^{i([\tau,t))}, \, t>\tau,
\end{eqnarray*}
\begin{eqnarray*}
U(t,\tau)=e^{A(t-\tau)} (I+B)^{-i([t,\tau))}, \, t<\tau,
\end{eqnarray*}
and $U(\tau,\tau)=I$. Moreover, for every $t\in \mathbb R$ and $j\in\mathbb Z$ the equation
\begin{eqnarray*}
U(t,\theta_j +) = U(t,\theta_j) (I+B)^{-1}
\end{eqnarray*} holds.
 If $(A2)$ additionally holds, then there exist real numbers $K\geq 1$ and $\alpha>0$ such that
\begin{eqnarray} \label{matriciantineq1}
\left\|U(t,\tau)\right\|\leq Ke^{-\alpha (t-\tau)}, \ t\geq \tau
\end{eqnarray}
 and
\begin{eqnarray} \label{matriciantineq2}
	\left\|U(t,\theta_j+)\right\|\leq Ke^{-\alpha (t-\theta_j)}, \  j\in \mathbb Z.
\end{eqnarray}
On the other hand, for every $t,\tau\in\mathbb R$ and $j\in\mathbb Z$, one can verify utilizing (\ref{propertyii}) that 
\begin{eqnarray} \label{cauchy11}
	U(t+\omega,\tau+\omega) = U(t,\tau) 
\end{eqnarray}
and
\begin{eqnarray} \label{cauchy12}
	U(t+\omega,(\theta_j+\omega)+) = U(t,\theta_j+). 
\end{eqnarray}

In the next section, we focus on the existence and uniqueness of a  bounded solution of system  (\ref{mainsyst}).

\section{Bounded Solutions} \label{secbddsol}

According to the results of the books \cite{Samoilenko95,Akhmet10} we have the following assertion.
\begin{lemma} \label{bddsolnlemma}
Suppose that the assumptions $(A1)-(A3)$ are fulfilled. A bounded function $\phi:\mathbb R \to \mathbb R^n$ is a solution of the impulsive system (\ref{mainsyst}) if  and only if the equation
\begin{eqnarray*}
\phi(t) = \displaystyle \int_{-\infty}^{t} U(t,s) \left[f(\phi(s), \phi(\gamma(s))) + g(s)\right] ds + \sum_{-\infty < \theta_j < t} U(t,\theta_j+) h(\phi(\theta_j))
\end{eqnarray*}
holds.
\end{lemma}

Let $d_{\infty}:\mathcal{PC}(\theta)\times \mathcal{PC}(\theta) \to [0,\infty)$ be the metric given by $$d_{\infty}\left(\phi,\psi \right)=\displaystyle \sup_{t\in\mathbb R} \left\| \phi(t) -\psi(t)\right\|,$$ and suppose that \begin{equation}\label{definitonsets}
 \mathcal{S} = \left\lbrace \phi \in \mathcal{PC}(\theta): \, \sup_{t\in\mathbb R} \left\|\phi(t) \right\| \leq \dfrac{K(M_f + M_g)}{\alpha} + \dfrac{KM_h p}{1-e^{-\alpha \omega}} \right\rbrace.
 \end{equation} 

\begin{lemma}
The metric space $\left( \mathcal{S}, d_{\infty}\right)$ is complete.
\end{lemma}

\noindent \textbf{Proof.} Because the metric space $\left( \mathcal{PC}(\theta), d_{\infty}\right)$ is complete \cite{Akhmet10}, it is enough to show that $\mathcal{S}$ is a closed subset of $ \mathcal{PC}(\theta)$. Suppose that a function $\phi$ belongs to the closure of $\mathcal{S}$. Then, there is a sequence $\left\lbrace \phi_n\right\rbrace$ of functions in $\mathcal{S}$ such that $\phi_n \to \phi$ as $n \to \infty$ uniformly on $\mathbb R$. For each $n\in\mathbb N$ and $t\in\mathbb R$, we have $$\left\| \phi(t)\right\| \leq \left\| \phi(t)-\phi_n(t)\right\| + \left\| \phi_n(t)\right\| \leq  \sup_{t\in\mathbb R}\left\| \phi(t)-\phi_n(t)\right\| + \dfrac{K(M_f + M_g)}{\alpha} + \dfrac{KM_h p}{1-e^{-\alpha \omega}}.$$
Owing to the uniform convergence we have $\displaystyle \sup_{t\in\mathbb R}\left\| \phi(t)-\phi_n(t)\right\| \to 0$ as $n \to \infty$. Thus, $$\sup_{t\in\mathbb R}\left\| \phi(t)\right\| \leq  \dfrac{K(M_f + M_g)}{\alpha} + \dfrac{KM_h p}{1-e^{-\alpha \omega}}.$$ The last inequality implies that $\phi\in\mathcal{S}$. Therefore, $\mathcal{S}$ is a closed subset of $\mathcal{PC}(\theta)$. Accordingly, the metric space $\left( \mathcal{S}, d_{\infty}\right)$ is complete. $\square$

The completeness of the metric $\left( \mathcal{S}, d_{\infty}\right)$ is used in the proof of the following theorem.

\begin{theorem} \label{lemmaboundedness}
If the assumptions $(A1)-(A5)$ hold, then the impulsive system (\ref{mainsyst}) possesses a unique bounded solution in $\mathcal{S}$.
\end{theorem}
\noindent \textbf{Proof.} Let $T$ be the operator defined on the metric space $\left( \mathcal{S}, d_{\infty}\right)$ such that
\begin{equation} \label{operatort}
(T \phi)(t) = \displaystyle \int_{-\infty}^{t} U(t,s) \left[f(\phi(s), \phi(\gamma(s))) + g(s)\right] ds + \sum_{-\infty < \theta_j < t} U(t,\theta_j+) h(\phi(\theta_j)).
\end{equation} 
In the proof, we will verify that $T$ has a unique fixed point in $\mathcal{S}$ using the Banach fixed point theorem. The fixed point of $T$ is the unique bounded solution of (\ref{mainsyst}) in accordance with Lemma \ref{bddsolnlemma}.

Let $\phi$ be an element of $\mathcal{S}$. Then, we have
\begin{equation*} \begin{aligned}
&\left\| (T \phi)(t)\right\|  \leq  \displaystyle \int_{-\infty}^{t} \left\| U(t,s)\right\|  \left\|  f(\phi(s), \phi(\gamma(s))) + g(s)\right\|   ds + \sum_{-\infty < \theta_j < t} \left\| U(t,\theta_j+)\right\| \left\|  h(\phi(\theta_j))\right\| \\
& \leq \displaystyle \int_{-\infty}^{t} K (M_f + M_g) e^{-\alpha(t-s)} ds +  \sum_{-\infty < \theta_j < t} KM_h e^{-\alpha(t-\theta_j)} \\
& < \dfrac{K(M_f + M_g)}{\alpha} + \dfrac{KM_h p}{1-e^{-\alpha \omega}}.
	\end{aligned}
\end{equation*}
Hence, $\displaystyle \sup_{t\in\mathbb R} \left\| (T \phi)(t)\right\| \leq \dfrac{K(M_f + M_g)}{\alpha} + \dfrac{KM_h p}{1-e^{-\alpha \omega}}$ so that $T\phi\in \mathcal{S}$. For that reason, we have $T\left(\mathcal{S} \right)\subseteq \mathcal{S}$.

Now, to confirm that $T$ is a contraction mapping, let us take two elements $\phi$ and $\psi$ of $\mathcal{S}$. One can verify for every $t\in\mathbb R$ that
\begin{equation*} \begin{aligned}
\left\| (T \phi)(t)-(T \psi)(t)\right\| &  \leq  \displaystyle \int_{-\infty}^{t} \left\| U(t,s)\right\|  \left\|  f(\phi(s), \phi(\gamma(s))) -f(\psi(s), \psi(\gamma(s))) \right\|   ds \\
& + \sum_{-\infty < \theta_j < t} \left\| U(t,\theta_j+)\right\| \left\|  h(\phi(\theta_j)) -h(\psi(\theta_j))\right\| \\
 &  \leq  \displaystyle \int_{-\infty}^{t} \left\| U(t,s)\right\|  \left\|  f(\phi(s), \phi(\gamma(s))) -f(\psi(s), \phi(\gamma(s))) \right\|   ds \\
  & + \displaystyle \int_{-\infty}^{t} \left\| U(t,s)\right\|  \left\|  f(\psi(s), \phi(\gamma(s))) -f(\psi(s), \psi(\gamma(s))) \right\|   ds \\
& + \sum_{-\infty < \theta_j < t} \left\| U(t,\theta_j+)\right\| \left\|  h(\phi(\theta_j)) -h(\psi(\theta_j))\right\| \\
 &  \leq  \displaystyle \int_{-\infty}^{t} KL_1 e^{-\alpha(t-s)}  \left\|   \phi(s)  - \psi(s)   \right\|   ds \\
& + \displaystyle \int_{-\infty}^{t}  KL_2 e^{-\alpha(t-s)}  \left\|   \phi(\gamma(s))  - \psi(\gamma(s))   \right\|  ds \\
& + \sum_{-\infty < \theta_j < t} K L_h e^{-\alpha(t-\theta_j)} \left\|  \phi(\theta_j) -\psi(\theta_j)\right\| \\
 &  \leq K \displaystyle\left(  \frac{L_1 + L_2}{\alpha}  +\frac{L_h p }{1-e^{-\alpha \omega}} \right) d_{\infty}(\phi,\psi).  
	\end{aligned}
\end{equation*}
Therefore,  $$ d_{\infty}\left( T \phi,T \psi  \right)  \leq K \displaystyle\left(  \frac{L_1 + L_2}{\alpha}  +\frac{L_h p }{1-e^{-\alpha \omega}} \right) d_{\infty}(\phi,\psi).$$
According to the assumption $(A5)$, $T$ is a contraction mapping. Since $\left( \mathcal{S}, d_{\infty}\right)$ is a complete metric space, there is a unique fixed point of $T$ in $\mathcal{S}$ by the Banach fixed point theorem. Consequently, system (\ref{mainsyst}) possesses a unique bounded solution in $\mathcal{S}$. $\square$

The next section is concerned with anti-periodic solutions of the impulsive system  (\ref{mainsyst}).

\section{Anti-periodic Solutions} \label{secantiper}

In this section, we discuss the existence and uniqueness of an anti-periodic solution of system (\ref{mainsyst}) based on the following definition.
\begin{definition} (\cite{Liu12}) \label{defn1}
 A function $\sigma\in\mathcal{PC}(\theta)$ is called anti-periodic if there exists a positive real number $\omega$ such that $\sigma(t+\omega) = - \sigma(t)$ for every $t\in\mathbb R$. The number $\omega$ is called an anti-period of $\sigma$.
\end{definition}

The subsequent assumption on the functions $f$ and $h$ in (\ref{mainsyst}) is needed for the verification of anti-periodicity.

\begin{itemize}
\item[\textbf{(A6)}] $f(-x,-y)=-f(x,y)$ and $h(-x)=-h(x)$ for every $x,y\in\mathbb R^n$.
\end{itemize}

The main result of this section is mentioned in the next theorem.
\begin{theorem} \label{antiperthm}
Suppose that the assumptions $(A1)-(A6)$ hold. If $g(t)$ is an anti-periodic function with anti-period $\omega$, then the unique bounded solution of the impulsive system (\ref{mainsyst}) is anti-periodic with anti-period $\omega$.
\end{theorem}

\noindent \textbf{Proof.} Let us denote by $\mathcal{AP}_{\omega}$ the set   of all anti-periodic functions in  $\mathcal{S}$, which is given by (\ref{definitonsets}), with anti-period $\omega$, i.e.,
$$\mathcal{AP}_{\omega} = \left\lbrace \phi\in \mathcal{S}: \ \phi(t+\omega) = - \phi(t) \ \textrm{for every} \  t\in\mathbb R \right\rbrace. $$ 
One can confirm that $\mathcal{AP}_{\omega}$ is a closed subset of $\mathcal{S}$ endowed with the metric $d_{\infty}$. Thus, the metric space $\left( \mathcal{AP}_{\omega}, d_{\infty}\right) $ is complete. We take into account the operator $T$ given by equation (\ref{operatort}), this time defined on the metric space $\left( \mathcal{AP}_{\omega}, d_{\infty}\right)$. A fixed point of $T$ in $\mathcal{AP}_{\omega}$ is an anti-periodic solution of system (\ref{mainsyst}) with anti-period $\omega$. Our aim is to demonstrate the existence of a unique fixed point of $T$ in $\mathcal{AP}_{\omega}$ via the Banach fixed point theorem.

In the proof of Lemma \ref{lemmaboundedness} we showed that $T$ is a contraction mapping. Therefore, it is sufficient to verify that $T(\mathcal{AP}_{\omega})\subseteq \mathcal{AP}_{\omega}$.

Fix an element $\phi$ of $\mathcal{AP}_{\omega}$. We have for every $t\in\mathbb R$ that 
\begin{equation*} \begin{aligned}
(T \phi)(t+\omega) &= \displaystyle \int_{-\infty}^{t+\omega} U(t+\omega,s) \left[f(\phi(s), \phi(\gamma(s))) + g(s)\right] ds + \sum_{-\infty < \theta_j < t+\omega} U(t+\omega,\theta_j+) h(\phi(\theta_j)) \\
&= \displaystyle \int_{-\infty}^{t} U(t+\omega,s+\omega) \left[f(\phi(s+\omega), \phi(\gamma(s+\omega))) + g(s+\omega)\right] ds   \\ & + \sum_{-\infty < \theta_j < t} U(t+\omega,(\theta_j+\omega)+) h(\phi(\theta_j+\omega))
\end{aligned} \end{equation*} 
If $s$ is a fixed real number, then there is an integer $j_s$ such that $\theta_{j_s}\leq s < \theta_{j_s+1}$. Therefore, $\theta_{j_s+p}\leq s+\omega  < \theta_{j_s+p+1}$. The equations $\gamma(s) =\zeta_{j_s}$ and $\gamma(s+\omega) =\zeta_{j_s+p}$ are valid according to the description of the piecewise constant argument $\gamma$. Thus, we have
\begin{equation} \label{propgamma}
\gamma(s+\omega)=\zeta_{j_s+p}=\zeta_{j_s}+\omega=\gamma(s)+\omega. 
\end{equation}
 Using the equations (\ref{cauchy11}), (\ref{cauchy12}), and (\ref{propgamma}) together with the assumption $(A6)$ we obtain that
\begin{equation*} \begin{aligned}
& (T \phi)(t+\omega)= \displaystyle \int_{-\infty}^{t} U(t,s) \left[f(-\phi(s), -\phi(\gamma(s))) - g(s)\right] ds   + \sum_{-\infty < \theta_j < t} U(t,\theta_j+) h(-\phi(\theta_j)) \\ 
& = - \displaystyle \int_{-\infty}^{t} U(t,s) \left[f(\phi(s), \phi(\gamma(s))) + g(s)\right] ds   - \sum_{-\infty < \theta_j < t} U(t,\theta_j+) h(\phi(\theta_j)) \\
& = - (T\phi)(t).
\end{aligned} \end{equation*} 
The last equation yields $T(\mathcal{AP}_{\omega})\subseteq \mathcal{AP}_{\omega}$. Since $T$ is a contraction mapping, it admits a unique fixed point in $\mathcal{AP}_{\omega}$. Consequently, there exists a unique anti-periodic solution of the impulsive system (\ref{mainsyst}) with   anti-period $\omega$ in accordance with the Banach fixed point theorem. $\square$

\section{Examples} \label{secexamples}

To obtain anti-periodic solutions generated by impulsive systems we take into account the function $g:\mathbb R \to \mathbb R$ defined by
\begin{eqnarray} \label{antipergfunc}
	g(t)=\begin{cases}  (t-6n)^3, & \mbox{ if } \ t \in [6n,\, 6n+1), \ n\in\mathbb Z,   \vspace{.15cm} \\ 1-\displaystyle \frac{(t-6n-1)^2}{4} , &  \mbox{ if } \ t \in [6n+1,\, 6n+3),  \ n\in\mathbb Z, 
		\vspace{.15cm} \\ -(t-6n-3)^3, &  \mbox{ if } \ t \in [6n+3,\, 6n+4),  \ n\in\mathbb Z,   
		\vspace{.15cm} \\ -1 + \displaystyle \frac{(t-6n-4)^2}{4}, &  \mbox{ if }  \ t \in [6n+4,\,  6n+6),  \ n\in\mathbb Z. 
	\end{cases}
\end{eqnarray} 
The function $g$ is anti-periodic with anti-period $3$ and it is bounded such that $\displaystyle \sup_{t\in\mathbb R} \left\| g(t)\right\|\leq 1 $. The graph of this function is depicted in Figure \ref{figure1}.
\begin{figure}[htbp]
	\centering
	\includegraphics[width=1.0\textwidth]{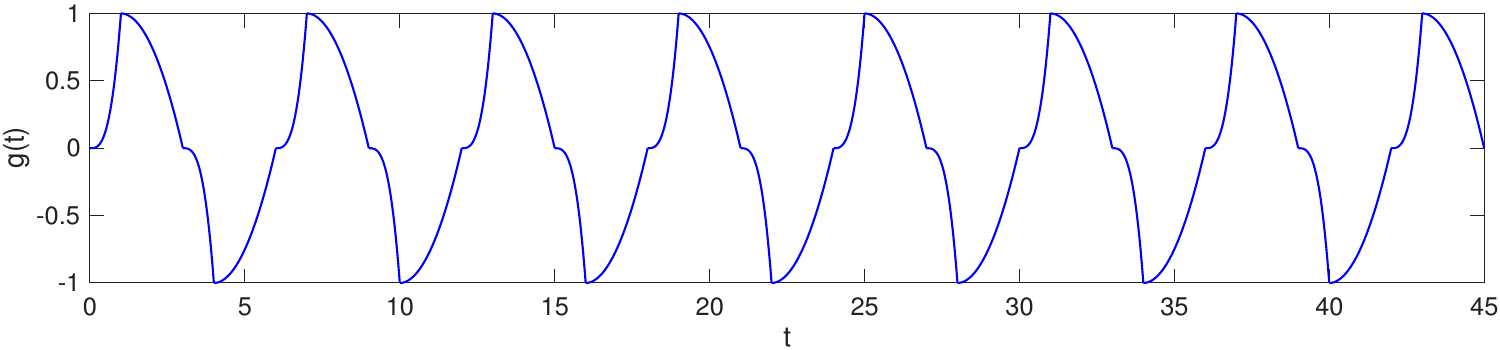}
	\caption{The graph of the function $g(t)$ defined by equation (\ref{antipergfunc}). It is seen in the figure that this function is anti-periodic with anti-period $3$. The boundedness is also observable.}
	\label{figure1}
\end{figure}

The purpose of the subsequent example is to demonstrate the distinctiveness of the present results compared to the papers \cite{Liu12,Liu13}.

\subsection{Example 1} 

Let us take into account the scalar impulsive differential equation
\begin{equation} \label{example1}
\begin{aligned}
& x'(t) +6x(t) =  0.25 \tanh(x(t)) + 0.15\tanh(x(\gamma(t))) + g(t), \ t\neq \theta_j, &\\	
& \Delta x|_{t=\theta_j} = -3x(\theta_j)  &  
\end{aligned}
\end{equation}
where $\theta_j=\displaystyle \frac{3}{2}j$, $\displaystyle \zeta_j= 1+\frac{3}{2}j$ for each $j\in\mathbb Z$ and $g(t)$ is the anti-periodic function defined by (\ref{antipergfunc}).

The assumptions of Theorem \ref{lemmaboundedness} and Theorem \ref{antiperthm} are fulfilled with $p=2$, $\omega=3$, $K=4$, $\alpha=6-\displaystyle \frac{2\ln 2}{3}$, $L_1=0.25$, and $L_2=0.15$. Accordingly, there is a unique bounded solution  of (\ref{example1}) and it is anti-periodic with anti-period $3$.

One can verify, on the other hand, that the scalar impulsive equation (\ref{example1}) is in the form  
\begin{equation} \label{example2}
	\begin{aligned}
		& x'(t) + a(t)x(t) = f(t,x(t), x(\gamma(t))) , \ t\neq \theta_j, \\
		& \Delta x|_{t=\theta_j} = J_j(x(\theta_j))  &  
	\end{aligned}
\end{equation}
with $a(t)=6$, $J_j(x)=-3x$, and $f(x_1,x_2) = 0.25 \tanh(x_1) + 0.15\tanh(x_2)+g(t)$. In this case, $x(x+J_j(x)) = -2x^2$, $J_j(x)(2x+J_j(x)) = 3x^2$, $xJ_j(x) = -3x^2$ for all $x \in \mathbb R$. The inequalities $x(x+J_j(x)) \geq 0$, $J_j(x)(2x+J_j(x)) \leq 0$, $x J_j(x) \geq 0$ were taken into account in the studies \cite{Liu12,Liu13} for the existence of at least one anti-periodic solution. However, in our case, none of them hold. This manifests that the hypotheses related to $J_j(x)$ mentioned in the papers \cite{Liu12,Liu13} are not fulfilled and the result of Theorem \ref{antiperthm} is novel.

In the next example, we focus on a $2$-dimensional impulsive system whose jump equation contains a nonlinear term.

\subsection{Example 2}  

In this example, we consider the impulsive system 
\begin{equation} \label{example3}
\begin{aligned}
& z_1'(t) = \displaystyle -\frac{13}{3} z_1(t) + \frac{5}{12} z_2(t)+ \frac{z_2}{10z_2^2+10} - \dfrac{1}{25} \sin(z_1(\gamma(t))) + 3g(t), \ t\neq \theta_j, &\\	
& z_2'(t) = \displaystyle \frac{4}{3} z_1(t) - \frac{17}{3} z_2(t)+ \frac{z_1}{20z_1^2+20} + \dfrac{1}{25} \sin(z_2(\gamma(t)))  -2g(t), \ t\neq \theta_j, &\\	
& \Delta z_1|_{t=\theta_j} = -\displaystyle \frac{2}{5} z_1(\theta_j) + \displaystyle \frac{1}{30} \arctan (z_1(\theta_j)), &  \\
& \Delta z_2|_{t=\theta_j} = -\displaystyle \frac{2}{5} z_2(\theta_j) - \displaystyle \frac{1}{35} \arctan (z_2(\theta_j)), &
\end{aligned}
\end{equation}
in which $\theta_j = \displaystyle \frac{3}{4}j + \frac{1}{8} (-1)^j$, $\zeta_j=  \displaystyle \frac{3}{8} + \frac{3}{4}j  $ for each $j\in\mathbb Z$ and $g(t)$ is the function defined by (\ref{antipergfunc}).
One can confirm that system (\ref{example3}) is in the form of (\ref{mainsyst}) with
\begin{equation*}
A= \begin{pmatrix} - \displaystyle \frac{13}{3}  &&  \displaystyle \frac{5}{12} \vspace{.2cm} \\ \displaystyle \frac{4}{3} && - \displaystyle \frac{17}{3}  \end{pmatrix}, \ B= \begin{pmatrix} -\displaystyle \frac{2}{5}  && 0 \\ 0 && -\displaystyle \frac{2}{5}  \end{pmatrix},
\end{equation*}
\begin{equation*}
f(x_1,x_2,y_1,y_2) = \displaystyle \begin{pmatrix} \displaystyle \frac{x_2}{10x_2^2+10} - \dfrac{1}{25} \sin(y_1) \vspace{.2cm}  \\ \displaystyle \frac{x_1}{20x_1^2+20} + \dfrac{1}{25} \sin(y_2)  \end{pmatrix}, \ h(x_1,x_2)=\begin{pmatrix} \displaystyle \frac{1}{30} \arctan(x_1) \vspace{.2cm}  \\ \displaystyle -\frac{1}{35} \arctan(x_2)  \end{pmatrix}.
\end{equation*}
The functions $f(x_1,x_2,y_1,y_2)$ and $h(x_1,x_2)$ are bounded, and the criteria $\theta_{j+p} = \theta_j+\omega$ and $\zeta_{j+p} = \zeta_j+\omega$ are valid for $p=4$ and $\omega=3$.
The eigenvalues of the matrix
\begin{equation*}
	A + \displaystyle \frac{p}{\omega} \textrm{Log}(I+B)= \begin{pmatrix} \displaystyle -\frac{13}{3}+ \frac{4}{3} \ln \left( \frac{3}{5}\right)  && \displaystyle \frac{5}{12} \\ \displaystyle \frac{4}{3}  && \displaystyle -\frac{17}{3}+ \frac{4}{3} \ln \left( \frac{3}{5}\right)   \end{pmatrix}
\end{equation*}
are the negative real numbers $-6+ \displaystyle \frac{4}{3}\ln \left( \frac{3}{5}\right)$ and $-4+ \displaystyle \frac{4}{3}\ln \left( \frac{3}{5}\right)$. The matriciant of the linear homogeneous impulsive system associated with (\ref{example3}) satisfies the equation
$$U(t,\tau) = \left(\displaystyle \frac{3}{5} \right)^{i([\tau,t))}P e^{D(t-\tau)} P^{-1}$$
for $t >\tau$, where $D=\textrm{diag}(-6,-4)$ and $$P=\begin{pmatrix} 1 && \displaystyle \frac{5}{4} \\ -4 && 1 \end{pmatrix}. $$
Accordingly, by taking $K=\left\|P \right\|\left\|P^{-1} \right\| \approx 2.91768$ and $\alpha=4$, it can be verified that the inequalities (\ref{matriciantineq1}) and (\ref{matriciantineq2}) hold. The assumptions $(A5)$ and $(A6)$ are also fulfilled for system (\ref{example3}) with $L_1=1/10$, $L_2=1/25$, and $L_h=1/30$. Therefore, there exists a unique solution of (\ref{example3}) which is bounded on the whole real axis by Theorem \ref{lemmaboundedness}. Moreover, this bounded solution is anti-periodic with anti-period $3$ in accordance with Theorem \ref{antiperthm}.

\section{Concluding Remarks} \label{secconcremark}
The present paper is concerned with anti-periodic solutions of impulsive systems with piecewise constant argument of generalized type. We also verify the existence and uniqueness of a bounded solution. Indeed, it is shown that the bounded solution is anti-periodic providing the function $g(t)$ in system (\ref{mainsyst}) has the same property. The presence of a nonlinear term with piecewise constant argument in the model constitutes the fundamental novelty. One of the examples is proposed to emphasize the distinctiveness of the new results. 

If we take into account the function $g(t)$ as an anti-periodic input for the impulsive system
\begin{eqnarray*} \label{concsystem}
	\begin{aligned}
		& z'(t) = A z(t) + f (z(t), z(\gamma(t))), \ t\neq \theta_j,    \\
		& \Delta z \big{|}_{t=\theta_j}= Bz(\theta_j) + h(z(\theta_j)),  
	\end{aligned} 
\end{eqnarray*}
then one can confirm that an anti-periodic output is generated. From the input-output point of view, we show that the dynamical behavior of the input is preserved by the impulsive system (\ref{mainsyst}).

The examples in Section \ref{secexamples} are constructed in such a way that the number $\omega$ in the criteria $\theta_{j+p}=\theta_j+\omega$, $\zeta_{j+p}=\zeta_j+\omega$, $j\in\mathbb Z$, and the anti-period of the function (\ref{antipergfunc}) are commensurable with a non-unit ratio.  This reveals that the number $p$ can be tuned to obtain a common value of the anti-period $\omega$ satisfying those criteria. 

In the future, our results can be extended for retarded impulsive differential equations \cite{Xia09} and systems of fractional differential equations with impulse actions \cite{Feckan12,Stamova16}.

\end{document}